\documentclass[preprint,12pt]{elsarticle}
\usepackage{amsmath}
\usepackage{amssymb}
\usepackage{amsthm}
\newtheorem{theorem}{Theorem}[section]
\newtheorem{lemma}[theorem]{Lemma}
\newtheorem{corollary}[theorem]{Corollary}
\newtheorem{example}[theorem]{Example}
\newtheorem{remark}[theorem]{Remark}
\newtheorem{definition}[theorem]{Definition}
\newcommand{\F}{\mathbb F_2}
\newcommand{\wt}{\operatorname{wt}}
\begin{document}
\begin{frontmatter}

\title{Certificate Complexity of Elementary Symmetric Boolean Functions of Arbitrary Degree}

\author[gsu]{Jing Zhang}
\ead{jzhang@govst.edu}

\author[wssu]{Yuan Li\corref{cor1}}
\ead{liyu@wssu.edu}
\cortext[cor1]{Corresponding author.}

\affiliation[gsu]{organization={Mathematics Department, Governors State University},
            city={University Park},
            state={Illinois},
            postcode={60484},
            country={USA}}
\affiliation[wssu]{organization={Department of Mathematics, Winston--Salem State University},
            city={Winston--Salem},
            state={North Carolina},
            postcode={27110},
            country={USA}}

\begin{abstract}
Let $\sigma_{n,d}$ denote the elementary symmetric Boolean function of $n$ variables and degree $d$. Previous work determined $C(\sigma_{n,d})$ when $d$ is odd or a power of two, but the general even non-power-of-two case remained open. We determine the certificate complexity for every degree $1\le d\le n$, thereby completing the classification for elementary symmetric Boolean functions. Writing $d=2^t m$ with $m$ odd, we obtain an explicit formula in which the possible deficit from the maximal value $n$ is controlled by $2^t$, while the exact value is determined by a binary containment condition involving $m$. In particular,
\[
n-2^{\nu_2(d)}+1\le C(\sigma_{n,d})\le n,
\]
and we characterize when the upper bound is attained. Moreover, we determine the least positive period of the certificate-complexity deficit $\Delta_d(n)=n-C(\sigma_{n,d})$: it is $1$ for odd $d$, equals $d$ when $d$ is a power of two, and equals $2^{\lfloor\log_2 d\rfloor+1}$ for even non-power-of-two $d$. This least-period problem is distinct from the classical periodicity of the underlying value sequence $\binom{j}{d}\bmod 2$. The known odd-degree and power-of-two formulas are recovered as special cases.
\end{abstract}


\end{frontmatter}


\section{Introduction}
Boolean functions are fundamental in switching theory, circuit and decision-tree complexity, cryptography, coding theory, and discrete dynamical systems. Symmetric Boolean functions form one of the most basic classes: their values depend only on Hamming weight. Their study goes back at least to Shannon's work on relay and switching circuits \cite{Shannon1938}; structural and cryptographic aspects were subsequently developed in, among many other works, \cite{ArnoldHarrison1963,CanteautVideau2005,CusickLiStanica2008,CusickLiStanica2009}.

Sensitivity, block sensitivity, certificate complexity, and decision-tree complexity are standard combinatorial measures of Boolean functions. Sensitivity arose in the study of parallel computation \cite{CookDworkReischuk1986}; Nisan developed block sensitivity and certificate methods in decision-tree complexity \cite{Nisan1991}. General relations among query-complexity measures are surveyed by Buhrman and de Wolf \cite{BuhrmanDeWolf2002}. More recent work has investigated these measures specifically for symmetric functions, including their separations and relations to adversary bounds and approximate degree \cite{MittalNairPatro2021}.

For elementary symmetric Boolean functions, binary arithmetic enters naturally. If $x\in\{0,1\}^n$ has Hamming weight $w$, then
\[
\sigma_{n,d}(x)=\binom{w}{d}\pmod2.
\]
Lucas' theorem \cite{Lucas1878} therefore converts the value sequence of $\sigma_{n,d}$ into a binary digit-containment problem. This viewpoint is standard in the literature on elementary symmetric Boolean functions and has been important in work on balancedness and cryptographic properties \cite{CanteautVideau2005,CusickLiStanica2008,CusickLiStanica2009}, as well as sensitivity and block sensitivity \cite{ZhangLiAdeyeye2021}.

Zhang, Li, and Adeyeye \cite{ZhangLiAdeyeye2021} obtained explicit results for sensitivity and block sensitivity of elementary symmetric Boolean functions. In our previous work \cite{ZhangLi2022}, we studied certificate complexity. We proved
\[
C(\sigma_{n,d})=n\quad(d\text{ odd})
\]
and, for $d=2^k$,
\[
C(\sigma_{n,2^k})=2^k\left\lfloor\frac n{2^k}\right\rfloor.
\]
In that paper, we left open the determination of certificate complexity for
\[
d=2^{k_1}+\cdots+2^{k_m},\qquad k_1>\cdots>k_m\ge1,\quad m\ge2.
\]
The present paper solves that open question completely.

Dahiya and Mahajan \cite{DahiyaMahajan2023} studied certificate complexity in connection with decision-tree rank. For a symmetric Boolean function, they characterized its certificate complexity by the shortest maximal interval of consecutive Hamming weights on which the function is constant. We use this known result explicitly. Thus, the constant-interval characterization itself is not new here. Our contribution is to combine that characterization with the special binary structure of the elementary symmetric Boolean function $\sigma_{n,d}$ supplied by Lucas' theorem. This yields an explicit formula for every degree $d$ and solves the open question left in our previous work \cite{ZhangLi2022}.

The resulting theorem applies to every degree. In particular, it closes the open family posed in our previous work \cite{ZhangLi2022} and gives short derivations of both previously known cases. It also yields structural information not contained in the special-case formulas: a sharp uniform bound on the deficit from $n$, a characterization of when the maximal value $C(\sigma_{n,d})=n$ occurs, and an exact determination of the least positive period of $\Delta_d(n)=n-C(\sigma_{n,d})$. The latter should not be confused with the classical periodicity of the underlying value sequence $\binom{j}{d}\bmod2$. The binary value sequence has period $2^{\lfloor\log_2 d\rfloor+1}$, a fact used previously in the study of elementary symmetric Boolean functions \cite{CanteautVideau2005,CusickLiStanica2008}. In contrast, the least period of the certificate-complexity deficit can collapse: it is $1$ for odd degrees and $d$ for power-of-two degrees, while the full binary period persists for even non-power-of-two degrees. To the best of our knowledge, this least-period behavior of the certificate-complexity deficit has not previously been determined.

\section{Preliminaries}

Let
\[
\F=\{0,1\},\qquad
\F^n=\{(x_1,\ldots,x_n):x_i\in\F,\ 1\le i\le n\}.
\]

\begin{definition}
For $x=(x_1,\ldots,x_n)\in\F^n$, the \emph{Hamming weight} of $x$,
denoted by $\wt(x)$, is the number of coordinates of $x$ that are equal
to $1$. Since $x_i\in\{0,1\}$,
\[
\wt(x)=x_1+\cdots+x_n.
\]
\end{definition}

\begin{definition}
A Boolean function $f:\F^n\to\F$ is called \emph{symmetric} if its
value depends only on the Hamming weight of the input. Thus, if
$\wt(x)=\wt(y)$, then $f(x)=f(y)$.
\end{definition}

For $1\le d\le n$, define the elementary symmetric Boolean function of
degree $d$ by
\[
\sigma_{n,d}(x_1,\ldots,x_n)
=\bigoplus_{1\le i_1<\cdots<i_d\le n}
x_{i_1}\cdots x_{i_d},
\]
where $\oplus$ denotes addition modulo $2$.

If $\wt(x)=j$, exactly $\binom{j}{d}$ monomials in the above sum have
value $1$. Therefore
\begin{equation}\label{eq:binomial}
\sigma_{n,d}(x)=\binom{\wt(x)}{d}\pmod2.
\end{equation}
For $0\le t\le n$, put
\[
\varepsilon_d(t)=\binom{t}{d}\pmod2.
\]
Thus $\varepsilon_d(t)$ is the value of $\sigma_{n,d}$ on every input
of Hamming weight $t$.

\begin{definition}
Let $f:\F^n\to\F$ and let $\alpha=(a_1,\ldots,a_n)\in\F^n$. A subset
$S=\{i_1,\ldots,i_k\}\subseteq\{1,\ldots,n\}$ is called a
\emph{certificate} of $f$ at $\alpha$ if fixing
\[
x_{i_1}=a_{i_1},\ldots,x_{i_k}=a_{i_k}
\]
forces the restricted function to be the constant $f(\alpha)$.
The \emph{certificate complexity} of $f$ at $\alpha$, denoted by
$C(f,\alpha)$, is the minimum cardinality of a certificate of $f$ at
$\alpha$. The certificate complexity of $f$ is
\[
C(f)=\max_{\alpha\in\F^n}C(f,\alpha).
\]
\end{definition}

For a symmetric Boolean function $f$, let
\[
f_j=f(x)\quad\text{whenever }\wt(x)=j,\qquad 0\le j\le n.
\]
The finite sequence
\[
f_0,f_1,\ldots,f_n
\]
will be called the \emph{Hamming-weight value sequence} of $f$.

\begin{definition}
A \emph{constant interval} of a symmetric Boolean function is an
interval of integers $[A,B]\subseteq[0,n]$ such that
\[
f_A=f_{A+1}=\cdots=f_B.
\]
It is \emph{maximal} if it cannot be extended to the left or to the
right while keeping the same value. Its \emph{length} is $B-A+1$.
Let $L_{\min}(f)$ denote the minimum length among all maximal constant
intervals of $f$.
\end{definition}

\begin{remark}
In our previous work with J.O. Adeyeye \cite{ZhangLiAdeyeye2021}, in which
we are the first two authors, we used the value vector of a symmetric
Boolean function and wrote it as consecutive runs of equal values. In the
terminology of Definition 2.4, each maximal run of equal entries in the
value vector corresponds exactly to a maximal constant interval, and the
number of entries in that run is the length of the interval.
\end{remark}

The following result is due to Dahiya and Mahajan
\cite[Lemma 5.9]{DahiyaMahajan2023}. Their notation
$\operatorname{Gap}_{\min}(f)$ is one less than the length used here.

\begin{lemma}[Dahiya--Mahajan]\label{lem:DM}
If $f:\F^n\to\F$ is a nonconstant symmetric Boolean function, then
\[
C(f)=n-L_{\min}(f)+1.
\]
\end{lemma}

For nonnegative integers $r$ and $s$, write
\[
r\preceq_2s
\]
if every binary digit that is $1$ in $r$ is also $1$ in $s$. For
example, $5=101_2\preceq_2 7=111_2$, whereas
$5\npreceq_2 6=110_2$.

We use Lucas' theorem modulo $2$. In the present notation it says
\begin{equation}\label{eq:lucas}
\binom{t}{d}\equiv1\pmod2
\quad\Longleftrightarrow\quad d\preceq_2t.
\end{equation}

\section{Certificate complexity for arbitrary degree}
\begin{theorem}\label{thm:main}
Let $1\le d\le n$. Write
\[
d=2^cD,\qquad D\text{ odd},\qquad h=2^c,
\]
and $n=hq+u$, $0\le u<h$. Then
\begin{equation}\label{eq:main}
C(\sigma_{n,d})=\begin{cases}
n-u,&D\preceq_2q\text{ or }D\preceq_2(q-1),\\[1mm]
n-h+1,&\text{otherwise}.
\end{cases}
\end{equation}
\end{theorem}
\begin{proof}
By \eqref{eq:binomial}, the sequence
\[
\varepsilon_d(0),\varepsilon_d(1),\ldots,\varepsilon_d(n)
\]
is exactly the Hamming-weight value sequence of $\sigma_{n,d}$.

Write
\[
j=ht+v,\qquad 0\le v<h.
\]
Since $d=hD=2^cD$, the last $c$ binary digits of $d$ are zero.
By Lucas' theorem,
\[
\varepsilon_d(ht+v)=1
\quad\Longleftrightarrow\quad
D\preceq_2t.
\tag{3.2}
\]
Thus, for each fixed $t$, as $v$ runs through
\[
0,1,\ldots,h-1,
\]
the $h$ consecutive Hamming weights
\[
ht,ht+1,\ldots,ht+h-1
\]
give the same value of $\varepsilon_d$, because the condition
$D\preceq_2t$ is independent of $v$. Hence the sequence
\[
\varepsilon_d(0),\varepsilon_d(1),\ldots
\]
is divided into blocks of $h$ consecutive terms, and every such block
is constant.

Since $D$ is odd, its last binary digit is $1$. If the $t$-block
has value $1$, then $D\preceq_2t$, and hence the last binary digit of
$t$ is $1$; in particular, $t$ is odd. Therefore both $t-1$ and $t+1$
are even, so their last binary digits are $0$. It follows that
\[
D\npreceq_2(t-1)\qquad\text{and}\qquad D\npreceq_2(t+1).
\]
Hence the blocks immediately before and after a $1$-block both have
value $0$. Consequently, every maximal constant interval having value
$1$ consists of exactly one $t$-block and therefore has exactly $h$
terms. On the other hand, two or more consecutive $t$-blocks may all
have value $0$, so they may join to form a longer maximal constant
interval having value $0$.

Now write $n=hq+u$, where $0\le u<h$. The last, possibly incomplete,
$q$-block is
\[
hq,hq+1,\ldots,hq+u
\]
and has length $u+1$.

Since two consecutive blocks cannot both have value $1$, there are only
three possibilities for the values of the $(q-1)$-block and the
$q$-block:
\[
(0,1),\qquad (1,0),\qquad (0,0).
\]

First suppose that the two blocks have values $(0,1)$ or $(1,0)$. In
either case, the value changes when we pass from the $(q-1)$-block to
the $q$-block. Hence a new maximal constant interval begins at $hq$,
and the final maximal constant interval is
\[
hq,hq+1,\ldots,hq+u.
\]
It has length $u+1$. Every earlier maximal constant interval has length
at least $h$, while $u+1\le h$. Therefore the shortest maximal constant
interval has length $u+1$. By Lemma \ref{lem:DM},
\[
C(\sigma_{n,d})=n-(u+1)+1=n-u.
\]
By (3.2), the cases $(0,1)$ and $(1,0)$ occur exactly when
\[
D\preceq_2q\quad\text{or}\quad D\preceq_2(q-1),
\]
respectively. This gives the first case of \eqref{eq:main}.

It remains to consider the possibility $(0,0)$. In this case the
$q$-block joins the preceding $0$-block, so the final maximal constant
interval contains the whole $(q-1)$-block together with the last
$q$-block. In particular, its length is greater than $h$. Since
$n\ge d$ and
\[
\varepsilon_d(d)=\binom dd=1,
\]
a complete maximal constant interval having value $1$ has already
occurred, and, as proved above, every such interval has exactly $h$
terms. Hence the shortest maximal constant interval has length exactly
$h$. By Lemma \ref{lem:DM},
\[
C(\sigma_{n,d})=n-h+1.
\]
By (3.2), the possibility $(0,0)$ is equivalent to
\[
D\npreceq_2q\quad\text{and}\quad D\npreceq_2(q-1),
\]
which is exactly the second case of \eqref{eq:main}. This completes the
proof.
\end{proof}

\begin{corollary}[Odd degree]\label{cor:odd}
If $d$ is odd, then $C(\sigma_{n,d})=n$.
\end{corollary}
\begin{proof}
Here $c=0$, $h=1$, and $u=0$, so both alternatives in \eqref{eq:main} equal $n$.
\end{proof}

\begin{corollary}[Power-of-two degree]\label{cor:power}
If $d=2^k$, then
\[
C(\sigma_{n,2^k})=2^k\left\lfloor\frac n{2^k}\right\rfloor.
\]
\end{corollary}
\begin{proof}
Here $D=1$ and $h=2^k$. Exactly one of $q$ and $q-1$ is odd, so $1\preceq_2q$ or $1\preceq_2(q-1)$ always holds. Therefore $C=n-u=hq$.
\end{proof}

These two corollaries agree with our previous results in \cite{ZhangLi2022}.

Theorem \ref{thm:main} also yields several structural consequences that are not visible from the two previously known special cases.

\begin{corollary}[Uniform distance from the maximum]\label{cor:bounds}
Let $h=2^{\nu_2(d)}$. For every $n\ge d$,
\[
n-h+1\le C(\sigma_{n,d})\le n.
\]
Equivalently,
\[
0\le n-C(\sigma_{n,d})\le h-1.
\]
\end{corollary}
\begin{proof}
Write $n=hq+u$, $0\le u<h$. By Theorem \ref{thm:main}, either
$C(\sigma_{n,d})=n-u$ or $C(\sigma_{n,d})=n-h+1$.
Since $0\le u\le h-1$, the result follows.
\end{proof}

\begin{corollary}[When the maximum value occurs]\label{cor:max}
If $d$ is odd, then $C(\sigma_{n,d})=n$ for every $n\ge d$. Suppose that $d$ is even, write $d=hD$ as in Theorem \ref{thm:main}, and write $n=hq+u$. Then
\[
C(\sigma_{n,d})=n
\]
if and only if $u=0$ and
\[
D\preceq_2 q\quad\text{or}\quad D\preceq_2(q-1).
\]
\end{corollary}
\begin{proof}
The odd case is Corollary \ref{cor:odd}. For even $d$ we have $h\ge2$, so the second alternative $n-h+1$ in Theorem \ref{thm:main} is strictly smaller than $n$. In the first alternative, $n-u=n$ holds exactly when $u=0$.
\end{proof}

\begin{corollary}[Periodicity of the certificate deficit]\label{cor:period}
Fix $d\ge1$ and put
\[
P=2^{\lfloor\log_2 d\rfloor+1}.
\]
Then for every $n\ge d$,
\[
(n+P)-C(\sigma_{n+P,d})=n-C(\sigma_{n,d}).
\]
Thus $P$ is a period of the deficit $n-C(\sigma_{n,d})$ (not necessarily its least period).
\end{corollary}
\begin{proof}
Write $d=2^cD=hD$, where $D$ is odd, and let
$m=\lfloor\log_2D\rfloor+1$. Then $P=h2^m$. If $n=hq+u$ with $0\le u<h$, then
\[
n+P=h(q+2^m)+u.
\]
The residue $u$ is unchanged. Moreover, all $1$-bits of $D$ occur among its lowest $m$ binary positions, and adding $2^m$ leaves those positions unchanged. Hence
\[
D\preceq_2q\iff D\preceq_2(q+2^m),
\]
and similarly
\[
D\preceq_2(q-1)\iff D\preceq_2(q+2^m-1).
\]
The same alternative of Theorem \ref{thm:main} therefore applies to $n$ and $n+P$, with the same deficit from the ambient dimension.
\end{proof}

The periodicity in Corollary~\ref{cor:period} is inherited from binary digit containment, but its least period need not equal the least period of the underlying value sequence. Indeed, the latter is known to have least period $2^{\lfloor\log_2 d\rfloor+1}$ \cite{CanteautVideau2005,CusickLiStanica2008}. The next theorem determines the least period of the derived certificate-complexity deficit itself. To the best of our knowledge, this exact least-period classification has not appeared previously.

\begin{theorem}[Least period of the certificate deficit]\label{thm:leastperiod}
For fixed $d\ge1$, let
\[
\Delta_d(n)=n-C(\sigma_{n,d}),\qquad n\ge d.
\]
The least positive period of $\Delta_d$ is
\[
\operatorname{per}(\Delta_d)=
\begin{cases}
1, & d \text{ is odd},\\[1mm]
d, & d \text{ is a power of two},\\[1mm]
2^{\lfloor\log_2 d\rfloor+1}, & d \text{ is even and not a power of two}.
\end{cases}
\]
\end{theorem}
\begin{proof}
If $d$ is odd, Corollary~\ref{cor:odd} gives $C(\sigma_{n,d})=n$, so
$\Delta_d(n)=0$ and its least positive period is $1$.

If $d=2^k$ with $k\ge1$, Corollary~\ref{cor:power} gives
\[
\Delta_d(n)=n-d\left\lfloor\frac nd\right\rfloor=n\bmod d.
\]
Thus $d$ is a period. To see that it is the least positive period, suppose
that $q$ were a period with $1\le q<d$. Since
\[
\Delta_d(d)=0,
\]
periodicity would require $\Delta_d(d+q)=0$. On the other hand,
\[
\Delta_d(d+q)=(d+q)\bmod d=q>0,
\]
a contradiction. Hence no positive integer smaller than $d$ is a period,
and the least positive period is $d$.

It remains to suppose that $d$ is even and not a power of two. Write
\[
d=hD,\qquad h=2^{\nu_2(d)},
\]
where $D>1$ is odd, and put
\[
R=2^{\lfloor\log_2D\rfloor}.
\]
Then
\[
P=2hR=2^{\lfloor\log_2d\rfloor+1}
\]
is a period by Corollary~\ref{cor:period}. We show that $P/2=hR$ is not a
period.

At $n=d=hD$ we have $u=0$ and $q=D$, so $D\preceq_2q$ and Theorem~\ref{thm:main}
gives $\Delta_d(d)=0$.

Now consider $n=d+hR=h(D+R)$. Again $u=0$, while $q=D+R$. Since
$R\le D<2R$, write $D=R+s$ with $1\le s<R$. Then
\[
q=D+R=2R+s,\qquad q-1=2R+s-1.
\]
The binary digit in the $R$-position is $1$ in $D=R+s$, but it is $0$
in both $2R+s$ and $2R+s-1$. Consequently
\[
D\npreceq_2q\qquad\text{and}\qquad D\npreceq_2(q-1).
\]
The second alternative of Theorem~\ref{thm:main} therefore applies, and
\[
\Delta_d(d+hR)=h-1>0.
\]
Thus $P/2$ is not a period.

Finally, the least positive period of a periodic sequence divides every
other period. Since $P$ is a power of two, any proper divisor of $P$ divides
$P/2$; if the least period were smaller than $P$, then $P/2$ would also be a
period, a contradiction. Hence the least period is $P$.
\end{proof}

Theorem \ref{thm:main} therefore not only settles degrees having any number of nonzero binary digits, but also shows that the deviation of certificate complexity from $n$ is controlled entirely by the $2$-adic part of the degree and follows a finite binary periodic pattern.

\begin{example}
For $d=6=2\cdot3$, $h=2$ and $D=3$. Thus, writing $n=2q+u$,
\[
C(\sigma_{n,6})=\begin{cases}n-u,&3\preceq_2q\text{ or }3\preceq_2(q-1),\\n-1,&\text{otherwise}.\end{cases}
\]
For instance, $C(\sigma_{6,6})=6$, $C(\sigma_{7,6})=6$, $C(\sigma_{8,6})=8$, and $C(\sigma_{9,6})=8$. Notice that at $n=8$ the final Hamming weight forms a one-term $0$-interval, which forces certificate complexity $8$.
\end{example}

\begin{example}
For $d=14=2\cdot7$, $h=2$ and $D=7$. At $n=14$, $q=7$, so $C(\sigma_{14,14})=14$. At $n=18$, $q=9$ and neither $9$ nor $8$ contains all three $1$-bits of $7$, hence $C(\sigma_{18,14})=17$.
\end{example}

\begin{remark}
The theorem gives an $O(\log n)$ bit test once $d$ and $n$ are written in binary. The quantity $2^{\nu_2(d)}$ controls the minimum complete run length, while the odd part $D=d/2^{\nu_2(d)}$ determines whether the final block begins a new run.
\end{remark}

\section{Conclusion}
We have determined $C(\sigma_{n,d})$ for all $1\le d\le n$, including the even non-power-of-two degrees not covered by our previous results \cite{ZhangLi2022}. Thus the certificate complexity of elementary symmetric Boolean functions is now explicitly determined for arbitrary degree. Beyond recovering the previously known odd-degree and power-of-two cases, the general formula exposes the binary structure responsible for the answer: the $2$-adic factor $2^{\nu_2(d)}$ bounds the possible deficit from $n$, while the odd part of $d$ determines the exact alternative through binary digit containment.

The formula also yields structural consequences that are not visible from the earlier special cases alone. In particular,
\[
n-2^{\nu_2(d)}+1\le C(\sigma_{n,d})\le n,
\]
with an explicit characterization of the degrees and values of $n$ for which $C(\sigma_{n,d})=n$. We also determine the least positive period of the deficit $n-C(\sigma_{n,d})$: it is $1$ for odd $d$, equals $d$ when $d$ is a power of two, and equals $2^{\lfloor\log_2 d\rfloor+1}$ for even non-power-of-two $d$. These results show that certificate complexity in this family is governed by a finite periodic pattern determined by the binary expansion of the degree. This completes the problem considered in \cite{ZhangLi2022} and provides a uniform framework for the certificate complexity of all elementary symmetric Boolean functions.

\section*{Declaration of competing interest}
The authors declare that they have no known competing financial interests or personal relationships that could have appeared to influence the work reported in this paper.

\section*{Declaration of generative AI and AI-assisted technologies in the manuscript preparation process}

During the preparation of this work, the authors used OpenAI ChatGPT to assist with
mathematical exploration, literature searching, organization of the manuscript, and
language and readability. The authors independently reviewed and verified the
mathematical arguments, references, and final text, edited the content as needed,
and take full responsibility for the content of the publication.


\begin{thebibliography}{99}

\bibitem{ArnoldHarrison1963} R. F. Arnold and M. A. Harrison, Algebraic properties of symmetric and partially symmetric Boolean functions, \emph{IEEE Trans. Electron. Comput.} EC-12 (1963), 244--251.

\bibitem{BuhrmanDeWolf2002} H. Buhrman and R. de Wolf, Complexity measures and decision tree complexity: a survey, \emph{Theoret. Comput. Sci.} 288 (2002), 21--43.

\bibitem{CanteautVideau2005} A. Canteaut and M. Videau, Symmetric Boolean functions, \emph{IEEE Trans. Inform. Theory} 51 (2005), 2791--2811.

\bibitem{CookDworkReischuk1986} S. A. Cook, C. Dwork, and R. Reischuk, Upper and lower time bounds for parallel random access machines without simultaneous writes, \emph{SIAM J. Comput.} 15 (1986), 87--97.

\bibitem{CusickLiStanica2008} T. W. Cusick, Y. Li, and P. St\u{a}nic\u{a}, Balanced symmetric functions over $\mathrm{GF}(p)$, \emph{IEEE Trans. Inform. Theory} 54 (2008), 1304--1307.

\bibitem{CusickLiStanica2009} T. W. Cusick, Y. Li, and P. St\u{a}nic\u{a}, On a conjecture for balanced symmetric Boolean functions, \emph{J. Math. Cryptol.} 3 (2009), 273--290.

\bibitem{DahiyaMahajan2023}
Y. Dahiya, M. Mahajan,
On (simple) decision tree rank,
Theoret. Comput. Sci. 978 (2023) 114177.
doi:10.1016/j.tcs.2023.114177.

\bibitem{Lucas1878} E. Lucas, Th\'eorie des fonctions num\'eriques simplement p\'eriodiques, \emph{Amer. J. Math.} 1 (1878), 184--196, 197--240, 289--321.

\bibitem{MittalNairPatro2021}
R. Mittal, S.S. Nair, S. Patro,
On query complexity measures and their relations for symmetric functions,
in: S. Kalyanasundaram, A. Maheshwari (Eds.),
Algorithms and Discrete Applied Mathematics, CALDAM 2024,
Lecture Notes in Computer Science, vol. 14508,
Springer, Cham, 2024, pp. 59--73,
doi:10.1007/978-3-031-52213-0\_5.

\bibitem{Nisan1991} N. Nisan, CREW PRAMs and decision trees, \emph{SIAM J. Comput.} 20 (1991), 999--1007.

\bibitem{Shannon1938} C. E. Shannon, A symbolic analysis of relay and switching circuits, \emph{Trans. AIEE} 57 (1938), 713--723.

\bibitem{ZhangLi2022} J. Zhang and Y. Li, Certificate complexity of elementary symmetric Boolean functions, \emph{Theoret. Comput. Sci.} 938 (2022), 16--23. DOI: 10.1016/j.tcs.2022.09.034.

\bibitem{ZhangLiAdeyeye2021} J. Zhang, Y. Li, and J. O. Adeyeye, Sensitivities and block sensitivities of elementary symmetric Boolean functions, \emph{J. Math. Cryptol.} 15 (2021), 434--453. DOI: 10.1515/jmc-2020-0042.

\end{thebibliography}
\end{document}